\documentclass[journal,12pt,onecolumn,draftclsnofoot,twoside]{IEEEtran}
\usepackage[T1]{fontenc}
\usepackage{amsmath,amssymb,mathtools}
\usepackage{microtype}
\usepackage{cite}
\usepackage[hidelinks]{hyperref}
\newtheorem{theorem}{Theorem}
\newtheorem{proposition}[theorem]{Proposition}
\newtheorem{lemma}[theorem]{Lemma}
\newtheorem{corollary}[theorem]{Corollary}
\newtheorem{remark}[theorem]{Remark}
\newcommand{\R}{\mathbb R}
\newcommand{\E}{\mathbb E}
\newcommand{\one}{\mathbf 1}
\newcommand{\dd}{\,\mathrm d}
\newcommand{\KL}{D_{\mathrm{KL}}}
\DeclareMathOperator{\intt}{int}
\DeclareMathOperator{\dom}{dom}
\DeclareMathOperator{\epi}{epi}
\DeclareMathOperator{\GL}{GL}
\allowdisplaybreaks[2]
\hypersetup{pdftitle={Sharp Rogers--Shephard Inequalities for R\'enyi Entropy}, 
 pdfauthor={Bang-Xian Han, Guoqing Zhan, Chiyu Zhou, and Zhuo-Nan Zhu},pdfsubject={Log-concave densities, R\'enyi differential entropy, sharp constants and equality cases}}
\title{Sharp Rogers--Shephard Inequalities\\for R\'enyi Entropy}
\author{Bang-Xian~Han, Guoqing~Zhan, Chiyu~Zhou, and Zhuo-Nan~Zhu%
\thanks{Bang-Xian Han is with the School of Mathematics, 
Shandong University, Jinan 250100, China
(e-mail: \mbox{hanbx@sdu.edu.cn}).}%
\thanks{Guoqing Zhan is with the School of Mathematical Sciences,
University of Science and Technology of China, Hefei 230026, China
(e-mail: \mbox{zhanguoqing@mail.ustc.edu.cn}).}%
\thanks{Chiyu Zhou is with the School of Mathematical Sciences,
University of Science and Technology of China, Hefei 230026, China
(e-mail: \mbox{dovong@mail.ustc.edu.cn}).}%
\thanks{Zhuo-Nan Zhu is with the School of Mathematical Sciences, 
University of Science and Technology of China, Hefei 230026, China
(e-mail: \mbox{zhuonanzhu@mail.ustc.edu.cn}).}}
\begin{document}
\maketitle
\begin{abstract}
We establish sharp weighted Rogers--Shephard inequalities for
R\'enyi differential entropy. For independent copies $X,X'$ of a log-concave random vector in $\R^n$, we prove $h_p(aX-bX')\le h_p(X)+n\log(a+b)$ for all $a,b>0$ and $1\le p\le\infty$, where orders $1$ and $\infty$ correspond to Shannon entropy and min-entropy, respectively. The Shannon case settles the entropic
Rogers--Shephard conjecture of Madiman and Kontoyiannis and its
higher-dimensional formulation. For $1\le p<\infty$, equality holds
precisely for invertible affine images of random vectors with
independent exponential coordinates. At $p=\infty$, the larger equality
class is characterized by possibly extended-valued convex potentials
that are positively homogeneous of degree one. For every $0<p<1$, uniform
laws on sufficiently high-dimensional simplices violate the unweighted
bound. Thus $p=1$ is the critical order for this bound to hold in
every dimension. The proof combines weighted convex intersections
with an exponential lift to a log-affine probability density. At the Shannon
endpoint, this lift yields an exact deficit decomposition with a
conditional mutual information term for an additive channel with
asymmetric Laplace noise. 
\end{abstract}


\section{Introduction}\label{sec:introduction}
\IEEEPARstart{W}{e} study entropy growth when an independent copy of a random
vector is subtracted from it. For a random vector $X$ with density
$f$ on $\R^n$, its Shannon differential entropy is
\[
 h(X)=h(f)=-\int_{\R^n}f(x)\log f(x)\dd x.
\]
Densities and integrals without an indicated measure are with
respect to Lebesgue measure. All logarithms are natural and we
use the convention $0\log 0=0$. This entropy has the opposite sign
to the functional $\int f\log f$ commonly used in optimal transport.

We ask for the sharp upper bound on $h(X-X')-h(X)$ when $X'$ is
an independent copy of a log-concave random vector $X$. In the
continuous setting, log-concavity means that $X$ has a log-concave
Lebesgue density; its support has nonempty interior. For a uniform
law on a convex body, the classical Rogers--Shephard inequality
\cite{RS} bounds the volume of the support of the difference.
Entropy depends on the entire convolution density, so a sharp
entropy bound requires more than control of the support.
This upper-bound problem complements the lower bounds supplied by
the entropy power inequality \cite{Shannon,Stam,Blachman}. We determine
the sharp upper bound under log-concavity and characterize all
laws attaining it.

The problem also has an additive-combinatorial background in
Ruzsa's sumset calculus and the study of additive representation
functions and additive energy. Inequalities for sumset and
difference-set sizes have entropy counterparts for independent
sums and differences. Madiman and Kontoyiannis \cite{MK} developed
this connection through the Ruzsa divergence for random vectors
and, more generally, random variables on locally compact Polish
abelian groups. More recently, Li, Gavalakis and Kontoyiannis
\cite{LGK} introduced entropic additive energy for continuous
random variables and established entropy inequalities for sums,
products, and sum-product combinations. This representation
viewpoint directs attention to the full difference distribution
and its entropy. We write $I(U;W)$ for mutual information:
if $U,W$ have joint
density $f_{U,W}$ and marginal densities $f_U,f_W$, then
\[
 I(U;W)=\int f_{U,W}(u,w)
 \log\frac{f_{U,W}(u,w)}{f_U(u)f_W(w)}\dd u\dd w,
\]
with zero integrand wherever $f_{U,W}=0$. For a log-concave random
vector $X$ and an independent copy $X'$, the Ruzsa divergence of
$X$ from itself is
\begin{equation}\label{eq:Ruzsa-definition}
 d_R(X\Vert X)=h(X-X')-h(X)=I(X-X';X'), 
\end{equation}
where the last equality in 
\eqref{eq:Ruzsa-definition} follows by conditioning on $X'$ and
using translation invariance of differential entropy.
The entropy power inequality yields
$d_R(X\Vert X)\ge(n/2)\log 2$.

Madiman and Kontoyiannis \cite[Corollaries 9 and 10]{MK} proved
$d_R(X\Vert X)\le2n\log 2$ for log-concave laws and conjectured
the sharp constant $\log 2$ in dimension one, with one-sided
exponentials as the only extremizers \cite[Conjecture 1]{MK}.
Melbourne and Tkocz \cite[Section VI]{MT} formulated the
dimension-$n$ conjecture and proved its R\'enyi counterpart for
orders $p\ge2$.

\subsection{The Shannon and Ruzsa bounds}

\begin{theorem}[Sharp Shannon bound]\label{thm:Shannon}
Let $X,X'$ be independent random vectors with the same log-concave
probability density on $\R^n$. Then
\begin{equation}\label{eq:Shannon}
 \frac n2\log 2\le h(X-X')-h(X)\le n\log 2.
\end{equation}
Equality in the lower bound holds if and only if $X$ is a
nondegenerate Gaussian random vector.
Equality in the upper bound holds if and only if $X$ is an
invertible affine image of a random vector with independent
one-sided exponential coordinates. 
Equivalently, $2\le N(X-X')/N(X)\le4$, where
$N(X)=(2\pi e)^{-1}e^{2h(X)/n}$ is the entropy power.
\end{theorem}

The lower bound and its equality cases follow from the entropy
power inequality, while the upper bound settles the conjecture
discussed above. In terms of Ruzsa divergence, this gives the
following corollary. 

\begin{corollary}[Ruzsa divergence]\label{cor:Ruzsa}
Under the assumptions of Theorem~\ref{thm:Shannon},
\begin{equation}\label{eq:Ruzsa-bound}
	\frac n2\log 2\le d_R(X\Vert X)=I(X-X';X')\le n\log 2,
\end{equation}
with the same equality cases as in Theorem~\ref{thm:Shannon}. 
The difference constant of \cite[Definition 4]{MK} therefore satisfies
\begin{equation}\label{eq:difference-constant}
 1\le\sigma_-(X):=\frac{N(X-X')}{2N(X)}\le2.
\end{equation}
\end{corollary}
This sharpens the upper bound on $d_R(X\Vert X)$ in \cite[Corollary 10]{MK}. 

\subsection{The full R\'enyi range and equality cases}
For a probability density $f$ on $\R^n$, the R\'enyi
differential entropy \cite{Renyi} of order
$p\in(0,1)\cup(1,\infty)$, together with its Shannon
and min-entropy limits, is given by
\begin{equation}\label{eq:Renyi-definition}
	\begin{aligned}
		h_p(f)
		= \frac{1}{1-p}\log\int_{\R^n} f(x)^p\dd x,\qquad h_1(f)
		= h(f), \qquad
		h_\infty(f)=-\log\|f\|_\infty.
	\end{aligned}
\end{equation} 
For a random vector $X$ with density $f$, we also write
$h_p(X)=h_p(f)$. We call $h_\infty$ the min-entropy. Except in the discrete discussion in
Section~\ref{sec:open-sums-discrete}, entropy means differential entropy.
Melbourne and Tkocz \cite[Theorem VI.1]{MT} proved
\begin{equation}\label{eq:known-range}
 h_p(X-X')\le h_p(X)+n\log 2,\qquad 2\le p\le\infty.
\end{equation}
Their argument combines the identity
\begin{equation}\label{eq:autocorrelation}
	h_\infty(X-X')=h_2(X)
\end{equation}
with sharp comparisons between R\'enyi entropies of log-concave
densities \cite{BM11,FMW,FLM}. For integrable functions $u,v$,
write $(u*v)(x)=\int u(x-y)v(y)\dd y$ for their convolution,
and let $\widetilde f(x)=f(-x)$ denote the reflection of $f$. By the Cauchy--Schwarz inequality, 
$(f*\widetilde f)(0)=\int f^2=\|f*\widetilde f\|_\infty$,
which proves \eqref{eq:autocorrelation}. 

For $0<p<2$, the same comparison method bounds the entropy increment
by $n\log p/(p-1)$, interpreted as $n$ at $p=1$. This Shannon bound
is due to Bobkov and Madiman \cite{BM13} and also follows from
\cite{BM11}. In dimension one, Melbourne, Nayar and Roberto
\cite[Corollary 2.3]{MNR} obtain the stronger bound
$h(X-X')-h(X)\le\frac12\log(4\pi/e)$ from their sharp lower
bound on entropy at fixed variance. The remaining problem is the
sharp bound for $1\le p<2$, including the Shannon endpoint.
We prove it in every dimension and allow arbitrary positive weights.

\begin{theorem}[Weighted R\'enyi bound]\label{thm:main}
Let $X,X'$ be independent random vectors with the same log-concave
probability density on $\R^n$. Then, for every $a,b>0$ and
$1\le p\le\infty$, 
\begin{equation}\label{eq:main}
 h_p(aX-bX')\le h_p(X)+n\log(a+b).
\end{equation}
All entropies in \eqref{eq:main} are finite. The constant is optimal
for each choice of $n,p,a,b$.
\end{theorem}

In particular, taking $a=b=1/2$, we obtain 
\[
h_p\left(\frac{X-X'}{2}\right)\le h_p(X),
\qquad 1\le p\le\infty, 
\]
which gives a symmetrization that preserves log-concavity and
does not increase R\'enyi entropy in this range. For $1\le p<\infty$, the equality cases
in \eqref{eq:main} are characterized as follows. 

\begin{theorem}[Finite-order equality]\label{thm:equality}
Under the assumptions of Theorem~\ref{thm:main}, fix $a,b>0$ and
$1\le p<\infty$. Equality holds in \eqref{eq:main} if and only if
\begin{equation}\label{eq:extremizers}
 \begin{gathered}
 X\stackrel{\mathrm d}=c+T(E_1,\ldots,E_n),\qquad 
 c\in\R^n,\qquad T\in\GL(n,\R),
 \end{gathered}
\end{equation}
where $\stackrel{\mathrm d}=$ denotes equality in distribution,
$\GL(n,\R)$ is the set of invertible real $n\times n$ matrices,
and $E_1,\ldots,E_n$ are independent standard exponential random
variables. 
\end{theorem}

Equivalently, writing $\one_A$ for the indicator of a set $A$,
$[v]_i$ for the $i$th coordinate of a vector $v$, and $\det T$ for
the determinant of $T$, the extremal densities are
\begin{equation}\label{eq:extremal-density}
 \begin{aligned}
 f(x)=\frac1{|\det T|}
 \exp\left(-\sum_{i=1}^n[T^{-1}(x-c)]_i\right)\,\one_{\{T^{-1}(x-c)\in[0,\infty)^n\}}.
 \end{aligned}
\end{equation}
Their supports are translates of simplicial cones, that is, cones
generated by $n$ linearly independent rays. At $p=\infty$
the equality class is larger: it consists of normalized densities
$e^{-\psi(x-c)}$ with $\psi$ convex and positively homogeneous of
degree one, allowing the value $+\infty$.
Proposition~\ref{prop:min-entropy-equality} gives the precise
statement and a direct proof for all positive weights. In particular, this class includes the two-sided Laplace
density $f(x)=\tfrac12e^{-|x|}$ on $\R$. 

\begin{theorem}[Critical R\'enyi order]\label{thm:critical}
For each fixed $0<p<1$, let $U,U'$ be independent random vectors
uniformly distributed on the same $n$-dimensional simplex. Then 
\begin{equation}\label{eq:failure}
 h_p(U-U')>h_p(U)+n\log 2
\end{equation}
for all sufficiently large $n$. Consequently, the bound $h_p(X-X')\le h_p(X)+n\log 2$
holds for all log-concave densities in all dimensions if and only
if $p\ge1$. 
\end{theorem}

Corollary~\ref{cor:signed-sums} extends the Shannon bound to signed
sums of any number of independent copies. At fixed total positive
and negative weights, entropy is maximized using at most one copy
of each sign, by Yu's entropy comparison argument \cite{Yu}.

Section~\ref{sec:interpolation} compares independent and quadratic
optimal transport couplings, that is, joint laws with prescribed
marginals, between a log-concave law and its reflection. For positive
weights summing to one, the independent interpolation has entropy
at most the common marginal entropy, whereas McCann's displacement
convexity \cite{McCann} gives the opposite bound for optimal transport. This comparison connects the present bound with sharp functional inequalities along Wasserstein geodesics and at Wasserstein barycenters \cite{HanZhuPoincare,HanLiuZhuBarycenter}. 

\subsection{Deficit and equality}
At the Shannon endpoint, passage to the limit proves the inequality
but does not determine its equality cases. An exact deficit
identity supplies the additional argument. Fix $a,b>0$ with $a+b=1$. For a log-concave probability density $f$,
write $\delta_1(f)=h(X)-h(aX-bX')$, where $X,X'$ are independent
random vectors with density $f$. Write $f=e^{-V}$, where $V=-\log f$ is convex, with value
$+\infty$ where $f=0$. 
Its effective domain is $\dom V=\{x:V(x)<\infty\}$.
Take independent standard exponential random variables
$E,E'$, independent of $X,X'$. The vector
$Y=(X,V(X)+E)$ has the log-affine density
\[
 F(x,t)=e^{-t}\one_{\{t\ge V(x)\}}.
\]
Set $D=aX-bX'$, $B=aV(X)-bV(X')$, and $L=aE-bE'$.
Then $aY-bY'=(D,B+L)$, where $Y'$ is defined from $(X',E')$.
We write $I(U;W\mid S)$ for conditional mutual information:
the average, over $S$, of the relative entropy between the
conditional joint law of $(U,W)$ and the product of its
conditional marginals \cite{CoverThomas}.
Section~\ref{sec:preliminaries} gives the precise
definitions, including for singular conditional laws.

\begin{theorem}[Shannon deficit decomposition]\label{thm:information}
With the preceding notation,
\begin{equation}\label{eq:intro-deficit}
 \delta_1(f)=\delta_1(F)+I(B;B+L\mid D).
\end{equation}
Both terms on the right are finite and nonnegative. The conditional
mutual information vanishes if and only if $V$ is affine on the
interior of its effective domain.
\end{theorem}

The second term is the conditional mutual information between
the input $B$ and output $B+L$ of an additive channel, given $D$.
Proposition~\ref{prop:deficit} proves this identity and the
corresponding R\'enyi deficit formula. 

\subsection{Proof strategy}
The proof has two steps. The first concerns log-affine densities
and estimates the power integrals of a difference density.
The guiding analogy is with additive representation functions
and additive energy. For a finite set $A$ in an abelian group,
the representation function
\[
 r_{A-A}(d)=\#\{(x,y)\in A\times A:x-y=d\}
\]
counts the representations of $d$ as a difference; $\#$ denotes
cardinality. The sum of squares $\sum_d r_{A-A}(d)^2$ is the
additive energy of $A$, which counts quadruples with equal
differences \cite[Lemma 2.9]{TaoVu}. For independent uniform
samples from $A$, the difference law is $r_{A-A}/|A|^2$, whose
squared masses sum to the normalized energy
$|A|^{-4}\sum_d r_{A-A}(d)^2$. In the continuous setting, the
difference density $g=f*\widetilde f$ plays the role of this
normalized representation function, and the corresponding energy
is $\int g^2$. R\'enyi entropy leads more generally to the power
integrals $\int g^p$. This analogy motivates the quantities we
study; log-concavity supplies the additional convex geometry
needed for the sharp bounds and equality cases.

For a log-affine density $f_C=Z^{-1}e^{-\ell}\one_C$, where $C$
is a closed convex set, $\ell:\R^n\to\R$ is linear, and
$Z=\int_C e^{-\ell}$ is positive and finite, the density of the
weighted difference is, up to normalization, an exponentially
weighted intersection mass. We introduce a one-parameter interpolation of these intersections
and apply the Pr\'ekopa--Leindler inequality to their weighted masses. Integration over the difference variable then
gives the sharp comparison of $p$th-power integrals for every
real $p>1$ (Section~\ref{sec:intersections}). The argument also
uses the log-concavity of $u^n\int_C e^{-u\ell}$ for $u>0$,
a special case of \cite[Theorem 2.9]{FMW}, with antecedents in
Borell's extension of Lyapunov's inequality \cite{BorellLyapunov}.
We include a proof by integration over convex slices. 

The second step removes the log-affine assumption. The lift
$Y=(X,V(X)+E)$ has a log-affine density on the convex epigraph
$\{(x,t):t\ge V(x)\}$. It is the logarithmic version of the
subgraph construction used in slice sampling
\cite[Sections 3 and 4]{Neal}. With $a+b=1$, let $g$ and $G$
be the densities of $D=aX-bX'$ and $aY-bY'$, respectively.
For every $p>1$, the comparison is
\[
 \frac1p\int f^p=\int F^p
 \le\int G^p\le\frac1p\int g^p.
\]
The first inequality is the log-affine case in dimension $n+1$.
For the second, let $k$ be the asymmetric Laplace density of
$L$ and let $\mu_d$ be the conditional law of $B$ given $D=d$.
The conditional density of $B+L$ given $D=d$ is
$(k*\mu_d)(u)=\int_{\R}k(u-v)\dd\mu_d(v)$. By Jensen's inequality, $\int(k*\mu_d)^p\le\int k^p=1/p$.
The factors $1/p$ cancel, proving the bound for all $p>1$.

The Shannon bound follows by letting $p\downarrow1$.
When the Shannon deficit vanishes,
\eqref{eq:intro-deficit} forces $\mu_d$ to be a point mass for
almost every $d$ with respect to the law of $D$, and hence $V$
to be affine on the interior of its domain. Dubuc's equality
theorem \cite{Dubuc} then reduces the support classification to
simplicial cones via the classification of finite-dimensional
vector lattices \cite{Yudin,AliprantisTourky}.

\subsection{Further related inequalities}
For a measurable set $A\subseteq\R^n$, write $|A|$ for its Lebesgue
volume. For sets $A,B\subseteq\R^n$, their Minkowski difference is
$A-B=\{x-y:x\in A,\ y\in B\}$. A convex body is a compact convex
set with nonempty interior. The geometric origin of the problem is
the Rogers--Shephard inequality \cite{RS}:
\begin{equation}\label{eq:classical-RS}
 |K-K|\le\binom{2n}{n}|K|
\end{equation}
for a convex body $K\subseteq\R^n$, with equality precisely for
simplices. The order-zero entropy is $h_0(f)=\log|\{x:f(x)>0\}|$.
For independent random vectors $X,X'$ uniformly distributed on $K$, this gives an order-zero entropy bound, since $h_0(X)=\log|K|$ and $h_0(X-X')=\log|K-K|$. Functional
versions based on supremal difference functions and convolution
bodies were developed in \cite{Colesanti,AGGJV}. For the connections between entropy inequalities and convex
geometry, including entropic Rogers--Shephard bounds, see
\cite{MMX}. 

Related convolution bounds include R\'enyi entropy power
inequalities \cite{BobkovChistyakov,MarsigliettiMelbourne}.
Sharp entropy bounds under moment constraints for symmetric
log-concave laws were proved in
\cite{MadimanNayarTkocz,BialobrzeskiNayar}. Bobkov and Madiman
\cite{BMreverse} established a reverse entropy power inequality
after separate volume-preserving linear transformations of the
summands. Ball, Nayar and Tkocz \cite{BNT} proved a reverse entropy
power inequality for one-dimensional projections of a log-concave
random vector, allowing dependence. For sums of independent copies,
Fu and Li \cite[Theorem 3.4]{FuLi} proved the sharp min-entropy bound
$h_\infty(X+X')-h_\infty(X)\le1$ for real-valued log-concave $X$,
with equality for one-sided exponentials. At $p=\infty$ in dimension
one, the optimal constant for the sum is therefore $1$, whereas
the difference bound \eqref{eq:known-range} has optimal constant $\log 2$.

Section~\ref{sec:preliminaries} contains the analytic preliminaries.
Sections~\ref{sec:intersections} and \ref{sec:extension} prove the
weighted inequality and the deficit identities.
Section~\ref{sec:equality} classifies the finite-order and
min-entropy equality cases.
Section~\ref{sec:Shannon-consequences} gives consequences for signed
sums and a comparison with displacement interpolation.
Section~\ref{sec:critical} proves failure below order one by
computing the difference densities for uniform laws on simplices.
Section~\ref{sec:open-problems} discusses extensions to different laws,
positive sums, discrete distributions, and orders below one.

\section{Preliminaries}\label{sec:preliminaries}
For a set $A$, write $\intt A$ for its interior and $\overline A$
for its closure. Set addition and scalar dilation mean
$A+B=\{x+y:x\in A,\ y\in B\}$ and $tA=\{tx:x\in A\}$;
in particular, $A+z=\{x+z:x\in A\}$ is translation by $z$.
We write $|x|$ for the Euclidean norm of a vector $x$.
Convex sets other than convex bodies may be unbounded.
We use upper semicontinuous representatives of log-concave densities
and identify densities equal almost everywhere. Their negative
logarithms are proper lower semicontinuous extended-real convex
functions. For a convex potential $V$, its epigraph is
$\epi V=\{(x,u):u\ge V(x)\}$. We write $\nabla V$ for its
gradient wherever it exists. The potential of a log-concave density
is finite and locally Lipschitz on the interior of its effective domain. Boundaries of full-dimensional convex sets have
Lebesgue measure zero. We use the usual facts about convex functions
and convex bodies as in \cite{Rockafellar,Schneider}.
For $p>0$, $L^p(\R^d)$ denotes the class of measurable functions
$u$ with $\int|u|^p<\infty$, identified up to equality almost
everywhere. We use $\E$ for expectation and $\E[\,\cdot\mid U]$
for conditional expectation given a random vector $U$.

The Pr\'ekopa--Leindler inequality \cite{Prekopa71,Leindler,Prekopa73}
states that, for $0<\theta<1$, if nonnegative measurable functions
$u,v,w$ have positive finite integrals and satisfy
\begin{equation}\label{eq:PL-assumption}
 w((1-\theta)x+\theta y)\ge u(x)^{1-\theta}v(y)^\theta,
 \qquad x,y\in\R^d,
\end{equation}
then
\begin{equation}\label{eq:PL}
 \int w\ge\left(\int u\right)^{1-\theta}
              \left(\int v\right)^\theta.
\end{equation}
In particular, integration in some variables preserves log-concavity;
see also \cite{Borell,BL}. We use Dubuc's equality theorem
\cite{Dubuc} in the following form: if equality holds in
\eqref{eq:PL}, the normalized functions $u/\int u$ and $v/\int v$
are translates almost everywhere; see also \cite[Theorem 1.1]{BFR}. Applied to $e^{-\ell}\one_B$ and $e^{-\ell}\one_C$, where $\ell$ is linear and $B,C$ are closed convex sets with nonempty interiors, it follows that $B$ and $C$ are translates as sets, since full-dimensional closed convex sets that agree almost everywhere are equal.

\begin{lemma}\label{lem:integrability}
For every log-concave probability density $f$ on $\R^d$, there exist constants $A,c>0$ such that
\begin{equation}\label{eq:envelope}
 f(x)\le Ae^{-c|x|}.
\end{equation}
Consequently, $\int f^p<\infty$ for every $p>0$,
$\int f|\log f|<\infty$, and the second moment is finite. 
The same holds for dilations, reflections, and convolutions of such
densities. Moreover,
\begin{equation}\label{eq:entropy-limits}
 \lim_{p\to1}h_p(f)=h(f),\qquad
 \lim_{p\to\infty}h_p(f)=h_\infty(f).
\end{equation}
\end{lemma}
\begin{IEEEproof}
The exponential bound \eqref{eq:envelope} is standard; see
\cite[Lemma~1]{CuleSamworth}. It implies the stated
integrability and moment assertions. Since $f\in L^p(\R^d)$ for
every $p>0$, the entropy limits \eqref{eq:entropy-limits} follow
from \cite[Lemma~V.1 and Corollary~V.2]{WangMadiman}.
Dilations and reflections preserve
log-concavity, while convolutions do so by the Pr\'ekopa--Leindler
inequality. Hence the same conclusions apply to these densities.
\end{IEEEproof}

R\'enyi entropy is additive for independent random vectors. For
an invertible linear map $T:\R^n\to\R^n$ and $c\in\R^n$,
\begin{equation}\label{eq:affine-entropy}
 h_p(TX+c)=h_p(X)+\log|\det T|,
 \qquad 0<p\le\infty.
\end{equation}
These identities follow by change of variables.

For random vectors $U,W$ with joint density $f_{U,W}$, their joint differential entropy is 
\[
 h(U,W):=h((U,W))=-\int f_{U,W}(u,w)\log f_{U,W}(u,w)\dd u\dd w.
\]
When the joint and marginal entropies are finite,
$I(U;W)=h(U)+h(W)-h(U,W)$.

We use mutual information and conditional mutual information for
possibly singular joint laws \cite[Sec. 5.5]{Gray}. For probability
measures $P,Q$, their relative entropy (also called Kullback--Leibler divergence)
\cite{KullbackLeibler} is
\[
 \KL(P\Vert Q)=\int\log\frac{\mathrm dP}{\mathrm dQ}\dd P
\]
if $P$ is absolutely continuous with respect to $Q$, and is
$+\infty$ otherwise; $\mathrm dP/\mathrm dQ$ is the Radon--Nikodym
derivative. The reference measure is the second argument of $\KL$. For probability densities $u,v$, we write
$\KL(u\Vert v)=\int u\log(u/v)$, with the integrand zero on
$\{u=0\}$ and value $+\infty$ if $u$ assigns positive mass to $\{v=0\}$.
Let $P_{U,W}$ denote the joint law of random vectors $U,W$, and
$P_U,P_W$ their marginal laws. Then
$I(U;W)=\KL(P_{U,W}\Vert P_U\otimes P_W)$, where $\otimes$
denotes the product of measures. This agrees with the density
definition in the introduction. For an additional random vector $S$, let
$P_{U,W\mid S=s}$ be a regular conditional joint law, with marginals
$P_{U\mid S=s},P_{W\mid S=s}$, and let $P_S$ be the law of $S$.
Conditional mutual information is
\[
 \begin{aligned}
 I(U;W\mid S)=\int \KL\!\left(P_{U,W\mid S=s}\,\Vert\,
                  P_{U\mid S=s}\otimes P_{W\mid S=s}\right)
                  \dd P_S(s).
 \end{aligned}
\]
It is nonnegative and vanishes exactly when $U,W$ are conditionally
independent given $S$.

When the joint
and marginal entropies are finite, the conditional differential
entropy is $h(W\mid U):=h(U,W)-h(U)$. We use the
standard chain rules for differential entropy and mutual
information \cite{CoverThomas}.

\section{Weighted intersections}\label{sec:intersections}
Throughout Sections~\ref{sec:intersections}--\ref{sec:equality},
unless otherwise stated, the weights satisfy
\begin{equation}\label{eq:normalization}
 a,b>0,\qquad a+b=1.
\end{equation}
For a log-concave probability density $f$, let $g$ be the density
of $aX-bX'$, where $X,X'$ are independent with density $f$.
For $1\le p\le\infty$, write
\begin{equation}\label{eq:normalized-deficit}
 \delta_p(f)=h_p(f)-h_p(g).
\end{equation}
The weights are held fixed in this notation, including when the
ambient dimension changes. By  \eqref{eq:affine-entropy}, the normalization loses no generality. For $p>1$, the desired inequality is
\begin{equation}\label{eq:moment-target}
 \int g^p\ge\int f^p.
\end{equation}

Let $C\subseteq\R^n$ be closed and convex with nonempty interior, and
let $\ell$ be linear, with
\begin{equation}\label{eq:log-affine}
 0<Z=\int_Ce^{-\ell(x)}\dd x<\infty,
 \qquad f_C=Z^{-1}e^{-\ell}\one_C.
\end{equation}
Here $X,X'$ are independent with density $f_C$, and $g$ denotes
the density of $aX-bX'$. Define
$Z(u)=\int_Ce^{-u\ell(x)}\dd x$ for $u>0$, so $Z=Z(1)$.
Lemma~\ref{lem:integrability} applied to $f_C$ bounds
$e^{-\ell}\one_C$ by an exponential envelope. Raising it to a
positive power proves finiteness of every $Z(u)$. Locally in $u>0$,
the same bound remains integrable after multiplication by
$|\ell(x)|$, so $Z$ is continuously differentiable and
\begin{equation}\label{eq:Z-derivative}
 (\log Z)'(1)=-\E\ell(X),\qquad
 h(f_C)=\log Z+\E\ell(X).
\end{equation}

The covariogram $d\mapsto|C\cap(C+d)|$ is the volume analogue
of the number of representations of a difference. It is used in
the classical Rogers--Shephard inequality \cite{RS} and its
convolution-body extensions \cite{AGGJV}. We replace volume by
exponential weight and allow the two sets to dilate at different
rates. The family below, with parameter $t\ge0$, has two useful
properties: at $t=0$ its
mass is the normalizing constant, and at $t=1$ it gives the
difference density up to normalization. Log-concavity in $t$
will turn these identities into a bound on the power integrals
of that density.
Set $r_t=b+at$ and $q_t=a+bt$ for $t\ge0$, and consider
\begin{equation}\label{eq:weighted-intersections}
 \begin{aligned}
 C_d(t)=\{x:bx+td\in r_tC,\ ax-td\in q_tC\},\qquad 
 H_d(t)=\int_{C_d(t)}e^{-\ell(x)}\dd x.
 \end{aligned}
\end{equation}
Here $d\in\R^n$ is the value of the weighted difference, and
$t$ is the parameter of the family.

\begin{lemma}\label{lem:intersections}
For each $d$, the function $t\mapsto H_d(t)$ is log-concave on
$[0,\infty)$ and the endpoint values are
\begin{equation}\label{eq:H-endpoints}
 H_d(0)=Z,\qquad H_d(1)=Z^2g(d).
\end{equation}
For $t>0$, $H_d(t)$ is bounded and continuous in $d$, and
\begin{equation}\label{eq:normalizer}
 \mathcal N(t):=\int_{\R^n}H_d(t)\dd d
 =\left(\frac{r_tq_t}{t}\right)^n Z(r_t)Z(q_t).
\end{equation}
Moreover, the function $u\mapsto\log(u^nZ(u))$ is concave for $u>0$.
\end{lemma}

The last assertion is the log-affine case of
\cite[Theorem 2.9]{FMW}. The proof uses the perspective construction
in \cite[Lemma 2.8]{FMW}. Related normalized
integral inequalities were proved by Berwald \cite{Berwald} and
Borell \cite{BorellLyapunov}. 
\begin{IEEEproof}
For a fixed $d$, the condition $bx+td\in(b+at)C$ is convex in
$(x,t)$: convex combinations reduce to convex combinations of points
of $C$, with weights proportional to their positive dilation
parameters. The second condition in \eqref{eq:weighted-intersections}
has the same property. The integrand is therefore jointly log-concave
in $(x,t)$, and Pr\'ekopa--Leindler gives the asserted log-concavity.

At $t=0$ both conditions say $x\in C$. At $t=1$, the substitution
$x=u+v$, $d=au-bv$ has Jacobian determinant of absolute value one and inverse
$u=bx+d$, $v=ax-d$. Linearity of $\ell$ then yields 
\eqref{eq:H-endpoints}. For the total integral at $t>0$, use
\[
 u=\frac{bx+td}{r_t},\qquad v=\frac{ax-td}{q_t}.
\]
The domain becomes $C\times C$, and the inverse map is
\begin{equation}\label{eq:inverse-map}
 x=r_tu+q_tv,\qquad d=\frac{ar_tu-bq_tv}{t}.
\end{equation}
The absolute value of its Jacobian determinant is $(r_tq_t/t)^n$, and its exponential weight
is $e^{-r_t\ell(u)}e^{-q_t\ell(v)}$. By Tonelli's theorem, we then obtain \eqref{eq:normalizer}.

If $U_t,V_t$ are independent with densities proportional to
$e^{-r_t\ell}\one_C$ and $e^{-q_t\ell}\one_C$, respectively,
\eqref{eq:inverse-map} identifies $H_d(t)/\mathcal N(t)$ as the
density of $(ar_t/t)U_t-(bq_t/t)V_t$. Both densities are bounded
and integrable. Their dilated convolution is bounded and continuous,
which proves the assertion concerning $d$. Finally,
\[
 u^nZ(u)=\int_{uC}e^{-\ell(y)}\dd y.
\]
The set $\{(y,u):u>0,\ y\in uC\}$ is convex. Applying
Pr\'ekopa--Leindler to these slices proves the last assertion.
\end{IEEEproof}

\begin{proposition}\label{prop:log-affine-moment}
If $f_C$ is as in \eqref{eq:log-affine}, then
\eqref{eq:moment-target} holds for every $p>1$.
\end{proposition}
\begin{IEEEproof}
	Since $H_d$ is log-concave and $H_d(0)=Z$, we have  
	\[
	H_d(p)\le Z^{1-p}H_d(1)^p.
	\] 
Using \eqref{eq:H-endpoints} and integrating, we obtain
\begin{equation}\label{eq:integrated-slices}
 \int g^p\ge\frac{\mathcal N(p)}{Z^{p+1}}
 =\left(\frac{r_pq_p}{p}\right)^n
       \frac{Z(r_p)Z(q_p)}{Z^{p+1}}.
\end{equation}
Write $A(u)=\log(u^nZ(u))$. Since $r_p=b\cdot1+a\cdot p$ and
$q_p=a\cdot1+b\cdot p$, by Lemma~\ref{lem:intersections}, we obtain 
\[
 A(r_p)\ge bA(1)+aA(p),\qquad
 A(q_p)\ge aA(1)+bA(p).
\]
Adding and exponentiating yields
\begin{equation}\label{eq:partition-bound}
 Z(r_p)Z(q_p)\ge
       \left(\frac p{r_pq_p}\right)^n Z\,Z(p).
\end{equation}
Combining \eqref{eq:integrated-slices} and \eqref{eq:partition-bound} then gives $\int g^p\ge Z(p)/Z^p=\int f_C^p$.
\end{IEEEproof}

\section{Exponential lifting and conditional mutual information}\label{sec:extension}
Let $X,X'$ be independent with a log-concave probability density
$f=e^{-V}$. The random variable $V(X)=-\log f(X)$ is the
information content of $X$ \cite{FMW}. The
subgraph identity $f(x)=\int_0^\infty\one_{\{s\le f(x)\}}\dd s$
represents $X$ as the first $n$ coordinates of a random point
uniformly distributed on $\{(x,s):0<s\le f(x)\}$.
The substitution $s=e^{-u}$ turns this into
an exponential density on $\epi V$. This is the same
auxiliary-variable construction, in logarithmic coordinates,
as in \cite[Sections 3 and 4]{Neal}. Earlier auxiliary-variable
representations appear in Edwards and Sokal \cite{EdwardsSokal}.
We use the representation to compare convolution entropies.
Take independent
standard exponential variables $E,E'$, independent of $X,X'$, and set
\begin{equation}\label{eq:extended-variables}
 Y=(X,V(X)+E),\qquad Y'=(X',V(X')+E').
\end{equation}
The density of $Y$ is
\begin{equation}\label{eq:epigraph-density}
 F(x,u)=e^{-u}\one_{\{u\ge V(x)\}},\qquad (x,u)\in\R^{n+1}.
\end{equation}
Indeed, given $X=x$, the last coordinate of $Y$ has density
$e^{V(x)-u}\one_{\{u\ge V(x)\}}$. Multiplying by
$f(x)=e^{-V(x)}$ yields \eqref{eq:epigraph-density}. Thus $F$ is a log-affine probability
density on a closed full-dimensional convex set, and
\begin{equation}\label{eq:epigraph-moment}
 \int_{\R^{n+1}}F^p=\frac1p\int_{\R^n}f^p,
 \qquad p>0.
\end{equation}

Put
\begin{equation}\label{eq:conditional-variables}
 \begin{gathered}
 D=aX-bX',\qquad L=aE-bE',\qquad 
 B=aV(X)-bV(X').
 \end{gathered}
\end{equation}
Then $aY-bY'=(D,B+L)$ and $L$ is independent of $(B,D)$.
Under \eqref{eq:normalization}, $L$ has the asymmetric Laplace density
\begin{equation}
 \begin{gathered}
 k(u)=
 \begin{cases}
 e^{-u/a},&u\ge0,\\
 e^{u/b},&u<0,
 \end{cases}
 \end{gathered}
\end{equation}
which implies 
\begin{equation}\label{eq:noise-density} 
	\int k^p=\frac1p,\qquad h(k)=1,\qquad \E L=a-b. 
\end{equation}
The density is symmetric when $a=b$.
Let $\mu_d$ be a regular conditional law of $B$ given $D=d$, and
write
\begin{equation}\label{eq:conditional-mixture}
 \begin{gathered}
 k_d(u)=\int k(u-v)\dd\mu_d(v),\qquad 
 G(d,u)=g(d)k_d(u).
 \end{gathered}
\end{equation}
Here $k_d$ is the conditional density of $B+L$ given $D=d$, and
$G$ is the joint density of $aY-bY'=(D,B+L)$; no density for
$\mu_d$ is assumed. The notation $k_D$ means this conditional
density evaluated at the random value $d=D$.

\begin{IEEEproof}[Proof of Theorem~\ref{thm:main}]
Fix $p>1$. By convexity of $z\mapsto z^p$ and Tonelli's theorem, we obtain 
\begin{equation}\label{eq:mixture-moment}
 \int k_d(u)^p\dd u
 \le\iint k(u-v)^p\dd u\dd\mu_d(v)=\frac1p.
\end{equation}
The density $F$ satisfies the hypotheses of
Proposition~\ref{prop:log-affine-moment} in dimension $n+1$.
Together with \eqref{eq:epigraph-moment} and
\eqref{eq:conditional-mixture}, it yields
\begin{equation}\label{eq:moment-sandwich}
 \frac1p\int f^p=\int F^p\le\int G^p\le\frac1p\int g^p.
\end{equation} 
This proves \eqref{eq:moment-target}. The limits in
Lemma~\ref{lem:integrability} give $p=1$ and $p=\infty$.
Applying the normalized inequality with coefficients
$a/(a+b)$ and $b/(a+b)$ and then scaling the difference by $a+b$
proves \eqref{eq:main} for arbitrary positive coefficients.

To verify sharpness, now let $a,b>0$ be arbitrary. For independent
standard exponentials $E,E'$, the density of $aE-bE'$ is
\begin{equation}\label{eq:extremal-noise}
 k_{a,b}(u)=\frac1{a+b}
 \begin{cases}
 e^{-u/a},&u\ge0,\\
 e^{u/b},&u<0.
 \end{cases}
\end{equation}
Consequently,
\[
 \int k_{a,b}^p=\frac{(a+b)^{1-p}}p,
 \qquad \int_0^\infty e^{-px}\dd x=\frac1p.
\]
Thus $h_p(aE-bE')-h_p(E)=\log(a+b)$ for all finite positive
orders, including $1$ by passage to the limit, and also for
$p=\infty$ by comparing maxima. Independent products and
\eqref{eq:affine-entropy} prove sharpness in every dimension and
sufficiency of \eqref{eq:extremizers} for equality.
\end{IEEEproof}

The next proposition expresses the deficit as a sum of two
nonnegative terms. The weights remain normalized as in
\eqref{eq:normalization}.

\begin{proposition}\label{prop:deficit}
For $1<p<\infty$, define
\[
 R_p(f)=\frac{p\int G^p}{\int g^p}.
\]
Then $0<R_p(f)\le1$ and
\begin{equation}\label{eq:Renyi-deficit}
 \delta_p(f)=\delta_p(F)-\frac1{p-1}\log R_p(f).
\end{equation}
At the Shannon endpoint,
\begin{equation}\label{eq:Shannon-deficit}
 \delta_1(f)=\delta_1(F)+I(B;B+L\mid D).
\end{equation}
The two summands on the right of either identity are finite and
nonnegative. In each identity the second summand vanishes if and
only if $V$ is affine on $\intt(\dom V)$.
\end{proposition}

\begin{IEEEproof}
The bound on $R_p(f)$ follows from \eqref{eq:mixture-moment},
and \eqref{eq:epigraph-moment} yields
\[
\frac{\int G^p}{\int F^p}
=\frac{\int g^p}{\int f^p}\,R_p(f).
\] 
Taking logarithms proves \eqref{eq:Renyi-deficit}. 
Theorem~\ref{thm:main} has already established $\delta_p(F)\ge0$.

For \eqref{eq:Shannon-deficit}, the entropy chain rule yields 
$h(Y)=h(X)+1$. The integrals defining $h(X)$, $h(Y)$, $h(D)$,
and $h(D,B+L)$ are absolutely convergent by
Lemma~\ref{lem:integrability}, also in dimension $n+1$.
Using the conditional density in \eqref{eq:conditional-mixture},
we can write the conditional mutual information as
\begin{equation}\label{eq:conditional-information}
 \begin{aligned}
 I(B;B+L\mid D)=\E\log\frac{k(L)}{k_D(B+L)}=h(B+L\mid D)-h(L).
 \end{aligned}
\end{equation}
Both logarithms are absolutely integrable. Indeed,
$\E|\log k(L)|=1$ by \eqref{eq:noise-density}, and
\[
\log k_D(B+L)=\log G(D,B+L)-\log g(D),
\]
with both terms on the right absolutely integrable by
Lemma~\ref{lem:integrability}.
Thus \eqref{eq:conditional-information} remains valid even when
the conditional law of $B$ given $D$ is singular.
Combining this identity with $h(Y)=h(X)+1$, $h(L)=1$,
and $aY-bY'=(D,B+L)$ yields \eqref{eq:Shannon-deficit}. 

For $p>1$, equality in \eqref{eq:mixture-moment} is equivalent to
$\mu_d$ being a point mass. Indeed, equality in Jensen's inequality
for the strictly convex function $z\mapsto z^p$, followed by
Fubini's theorem, gives
$k(\,\cdot-v)=k(\,\cdot-w)$ almost everywhere for
$\mu_d\otimes\mu_d$-almost every $(v,w)$. Distinct translates
cannot coincide: their finite means differ by $v-w$.
Thus $v=w$ almost surely, which forces $\mu_d$ to be a point mass.
The converse is immediate. If $R_p(f)=1$, the nonnegative difference between the two sides of
\eqref{eq:mixture-moment} has zero integral against $g(d)^p\dd d$.
This measure has the same null sets as $g(d)\dd d$ on $\{g>0\}$.
 It follows that $B$ is almost surely a measurable function of $D$. 

At $p=1$, zero conditional mutual information means conditional
independence of $B$ and $B+L$ given $D$. The variables are integrable,
since $\E|V(X)|=\int f|\log f|<\infty$. With $\E L=a-b$,
\[
 \begin{aligned}
 B+(a-b)=\E[B+L\mid B,D]=\E[B+L\mid D]=\E[B\mid D]+(a-b), 
 \end{aligned}
\]
which again implies $B$ is almost surely a measurable function of $D$. 

It remains to show that $B$ is almost surely a measurable  function of $D$ if and only if $V$ is affine on $\intt(\dom V)$. Suppose that
\begin{equation}\label{eq:potential-equation}
 aV(X)-bV(X')=\varphi(aX-bX')\qquad\text{almost surely}
\end{equation}
for a measurable function $\varphi:\R^n\to\R$.
On $O=\intt(\dom V)$ the density $f$
is strictly positive. After the invertible linear change of variables
$(x,y)\mapsto(d,y)=(ax-by,y)$, for almost every $d$, by Fubini's theorem, we obtain 
\[
 aV((d+by)/a)-bV(y)=\varphi(d)
\]
for almost every $y$ in the open convex set $O\cap(aO-d)/b$.
Continuity of $V$ extends the equality throughout that set.
Differentiating at almost every $y$ gives
$\nabla V((d+by)/a)=\nabla V(y)$. Transforming back and applying Fubini's theorem again, $\nabla V(x)=\nabla V(y)$ for almost every $(x,y)\in O^2$. Hence $\nabla V$ is constant almost everywhere on $O$. Then by local Lipschitz regularity and connectedness, $V$ is affine on $O$. Conversely, if $V(x)=\ell(x)+c$ there, with $\ell$ linear and $c\in\R$, then 
$B=\ell(D)+(a-b)c$, and both remainder terms vanish.
\end{IEEEproof}

In particular, \eqref{eq:main} is strict for every $1\le p<\infty$
when the potential is not affine on the interior of its domain.
The remaining equality condition concerns the support of a
log-affine density.

\section{Equality cases}\label{sec:equality}
We retain the normalized positive weights \eqref{eq:normalization}. 
The following estimate retains a concavity gap at a fixed
interpolation parameter. A vanishing Shannon deficit then
forces equality in the Pr\'ekopa--Leindler inequality. 

\begin{proposition}\label{prop:Shannon-gap}
For $f_C$ as in \eqref{eq:log-affine}, let $\Omega=\{d:g(d)>0\}$.
For $0\leq s\leq 1$ define the concavity gap on $\Omega$ by
\[
 \Gamma_s(d)=\log H_d(s)-(1-s)\log Z-s\log H_d(1).
\]
Then for every fixed $0<\sigma<1$,
\begin{equation}\label{eq:retained-gap}
 \delta_1(f_C)\ge\frac1{1-\sigma}
                      \int_\Omega g(d)\Gamma_\sigma(d)\dd d\ge0.
\end{equation}
\end{proposition}
\begin{IEEEproof}
By the log-concavity of $H_d$, $s\mapsto\Gamma_s(d)$ is
nonnegative and concave on $[0,1]$. Since $g$ and $H_d(s)$ are continuous
in $d$, $\Omega$ is open and $\Gamma_s$ is continuous there.
For $0<s<1$, define the probability density
$\rho_s(d)=H_d(s)/\mathcal N(s)$. Expanding $\KL(g\Vert\rho_s)$ and using
$H_d(1)=\mathcal N(1)g(d)$, we obtain 
\begin{equation}\label{eq:finite-Shannon}
 \begin{aligned}
 \frac{\int_\Omega g\Gamma_s+\KL(g\Vert\rho_s)}{1-s}=\frac{\log\mathcal N(s)-s\log\mathcal N(1)}{1-s}
                       -\log Z-h(g).
 \end{aligned}
\end{equation} 
For fixed $0<s<1$, $H_d(s)$ is bounded above uniformly in $d$,
while log-concavity gives
\[
\log H_d(s)\ge(1+s)\log Z+s\log g(d),
\qquad d\in\Omega.
\]
Since $\int g|\log g|<\infty$ by
Lemma~\ref{lem:integrability}, these bounds imply
$\int_\Omega g|\log H_d(s)|<\infty$. Thus all terms in \eqref{eq:finite-Shannon} are finite. 

Note that equations \eqref{eq:normalizer} and \eqref{eq:Z-derivative}
give $\mathcal N(1)=Z^2$ and $(\log\mathcal N)'(1)=-\E\ell(X)$. By \eqref{eq:Z-derivative}, the right side of \eqref{eq:finite-Shannon} tends to
$\log Z+\E\ell(X)-h(g)=\delta_1(f_C)$ as $s\uparrow1$.
For $\sigma\le s<1$, by concavity, 
\[
 \Gamma_s(d)\ge\frac{1-s}{1-\sigma}\Gamma_\sigma(d).
\]
Using this estimate and $\KL(g\Vert\rho_s)\ge0$ in \eqref{eq:finite-Shannon}, we obtain \eqref{eq:retained-gap} by letting $s\uparrow1$. 
\end{IEEEproof}

The next lemma is a geometric form of Yudin's
classification of finite-dimensional Archimedean vector lattices
\cite{Yudin}; see \cite{AliprantisTourky,Schaefer} for modern
accounts. For a cone $K$, pointed means $K\cap(-K)=\{0\}$,
and generating means $K-K=\R^n$. A compact base is a compact
convex section meeting each nonzero ray of $K$ exactly once.
A closed pointed generating cone defines an Archimedean order,
and principal intersections make that order a lattice.
We include a proof using the Riesz decomposition property, so
the geometric part of the equality argument is self-contained.

\begin{lemma}\label{lem:simplicial}
Let $K\subseteq\R^n$ be a closed pointed convex cone with nonempty
interior and a compact base. Suppose that
\begin{equation}\label{eq:principal-intersections}
 K\cap(z+K)\text{ is a translate of }K
 \qquad\text{for every }z\in\R^n. 
\end{equation}
Then $K$ is generated by $n$ linearly independent rays. 
\end{lemma}
\begin{IEEEproof}
Define $x\le_K y$ when $y-x\in K$. The common upper bounds of any
$x,y$ form $(x+K)\cap(y+K)$, which is a principal upper set by
\eqref{eq:principal-intersections}. The pair has a least upper
bound $x\vee y$, unique by pointedness, and a greatest lower bound
$x\wedge y=-((-x)\vee(-y))$.

If $0\le_Kw\le_Ku+v$ with $u,v\in K$, set
\[
 w_1=w\wedge u,\qquad w_2=w-w_1=(w-u)\vee0.
\]
Then $0\le_K w_1\le_K u$. Since $v$ is a common upper bound
of $w-u$ and $0$, we also have $0\le_K w_2\le_K v$.
Together with $w=w_1+w_2$, this proves the Riesz decomposition
property \cite{AliprantisTourky,Schaefer}. The corresponding
decomposition for any finite sum follows by induction. 

Two different extreme rays have no common nonzero positive lower
bound: such a vector would lie on both rays. Suppose a finite family
$v_1,\ldots,v_m$ of distinct extreme-ray generators satisfies a
nontrivial linear relation. A strictly positive functional defining
the compact base lets us write it using disjoint nonempty index
sets $I,J$ and positive coefficients $\alpha_i,\beta_j$ as
\[
 \sum_{i\in I}\alpha_i v_i=\sum_{j\in J}\beta_jv_j,
 \qquad \alpha_i,\beta_j>0,\qquad  I\cap J=\varnothing.
\]
Fix $i\in I$. By the Riesz decomposition property,
$\alpha_i v_i=\sum_{j\in J}w_j$, where
$0\le_K w_j\le_K\beta_jv_j$.
Since all summands lie in $K$, each $w_j$ is also below
$\alpha_i v_i$. The rays generated by $v_i$ and $v_j$ are distinct, so every $w_j$ is zero, contradicting $\alpha_i v_i\ne0$. Thus any finite set of distinct extreme rays is linearly independent, and there are at most $n$ of them. A compact base has affine dimension $n-1$ and is the convex hull of
its extreme points \cite{Schneider}. It therefore has at least $n$
extreme points, which correspond to the extreme rays. There are
exactly $n$, and they generate $K$.
\end{IEEEproof}

\begin{proposition}\label{prop:log-affine-equality}
For $f_C$ as in \eqref{eq:log-affine} and $1\le p<\infty$,
$\delta_p(f_C)=0$ if and only if $C$ is a translate of a
full-dimensional simplicial cone.
\end{proposition}
\begin{IEEEproof}
We first show that $C$ is a translate of a cone, and then use the
intersection identities to prove that this cone is simplicial.
Suppose $p>1$. Equality in
\eqref{eq:integrated-slices} and \eqref{eq:partition-bound} forces
equality in each of the two concavity comparisons for
$A(u)=\log(u^nZ(u))$. In particular,
\begin{equation}\label{eq:partition-equality}
 A(r_p)=bA(1)+aA(p).
\end{equation} 
By convexity of $C$, $bC+apC=r_pC$. Thus
\eqref{eq:partition-equality} expresses equality in the
Pr\'ekopa--Leindler inequality for $e^{-\ell}\one_C$
and $e^{-\ell}\one_{pC}$ at parameter $a$.
Dubuc's equality theorem then yields $pC=C+v$ for some $v\in\R^n$. 

Now suppose $p=1$. Proposition~\ref{prop:Shannon-gap} gives
$\int_\Omega g\Gamma_{1/2}=0$, which implies
$\Gamma_{1/2}(d)=0$ for every $d\in\Omega$. Convexity of the slice set yields $(C+C_d(1))/2\subseteq C_d(1/2)$. Since $\ell$ is
linear, the corresponding weighted indicators satisfy the
Pr\'ekopa--Leindler hypothesis. The identity $\Gamma_{1/2}(d)=0$
gives $H_d(1/2)=Z^{1/2}H_d(1)^{1/2}$, which is equality in the
integral inequality. Dubuc's theorem therefore yields 
\begin{equation}\label{eq:Shannon-slice-translation}
 C_d(1)\text{ is a translate of }C,\qquad d\in\Omega.
\end{equation}
Translate the density so that $0\in\intt C$, which preserves its deficit. Then $g(0)>0$ and, since $0\in C$ and $C$ is convex, it holds that 
\[
C_0(1)=b^{-1}C\cap a^{-1}C
=\frac1{\max(a,b)}C.
\] 
The dilation factor is strictly greater than one. In both cases, we have proved
\begin{equation}\label{eq:dilated-translate}
 \lambda C=C+v\qquad\text{for some }\lambda>1.
\end{equation}

Put $c=v/(\lambda-1)$ and $K=C-c$. Then $\lambda K=K$.
Contraction by powers of $\lambda$, followed by closedness, yields 
$0\in K$. Convexity and invariance under these dilations imply
$tK\subseteq K$ for every $t\ge0$. Thus $K$ is a closed convex cone with nonempty interior. We claim that 
\begin{equation}\label{eq:cone-positivity}
	\ell(w)>0,\qquad w\in K\setminus\{0\}.
\end{equation} 
Indeed, if $\ell(w)\le0$ for some $w\in K\setminus\{0\}$,
translating an interior ball of $K$ by sufficiently separated
positive multiples of $w$ produces pairwise disjoint balls in $K$
whose weighted masses do not decrease, contradicting
$\int_K e^{-\ell}<\infty$. Hence $K$ is pointed.
Moreover, $K\cap\{\ell=1\}$ is compact: otherwise, normalizing
an unbounded sequence in this section and passing to a subsequence
would yield a unit vector in $K$ on which $\ell$ vanishes,
contradicting \eqref{eq:cone-positivity}. 

After translating the density to $K$, apply the equality comparisons
again. Since $\intt K-\intt K=\R^n$ and $a,b>0$, its weighted
difference density is positive everywhere. If $p>1$, equality in the
integrated slice comparison, with continuity in $d$, yields
\[
 H_d(p)=Z^{1-p}H_d(1)^p, \qquad d\in\R^n. 
\]
Equality in Pr\'ekopa--Leindler at parameter $1/p$ implies that
$C_d(p)$ is a translate of $K$. On a cone, the definition simplifies
to
\begin{equation}\label{eq:cone-slices}
 C_d(t)=(K-td/b)\cap(K+td/a)=tC_d(1),\qquad t>0.
\end{equation}
It follows that $C_d(1)$ is a translate of $K$. If $p=1$, the same
conclusion follows directly from \eqref{eq:Shannon-slice-translation},
now valid for every $d$. Translating \eqref{eq:cone-slices} at $t=1$
by $d/b$ shows that $K\cap(K+d/(ab))$ is a translate of $K$. Thus, Lemma~\ref{lem:simplicial} proves that $K$ is simplicial. 

Conversely, integrability on a translated simplicial cone forces
positive rates along its generators, and the calculation
\eqref{eq:extremal-noise} gives equality.
\end{IEEEproof}

\begin{IEEEproof}[Proof of Theorem~\ref{thm:equality}]
By the scaling property of entropy, we may assume that $a+b=1$. If $\delta_p(f)=0$ for a finite $p\ge1$,
Proposition~\ref{prop:deficit} forces the second term in the
corresponding deficit identity to vanish. Hence $V$ is affine on
$O=\intt(\dom V)$. Since the boundary of $\overline O$ has
Lebesgue measure zero, $f$ agrees almost everywhere with a density
of the form $Z^{-1}e^{-\ell}\one_C$ on $C=\overline O$. Then by Proposition~\ref{prop:log-affine-equality}, we obtain $C=c+K$, where $K$ is a full-dimensional simplicial cone.
If its independent generators are $v_1,\ldots,v_n$, set
$\lambda_i=\ell(v_i)>0$. In the coordinates
$x=c+\sum_i t_iv_i$, with $t_i\ge0$, the normalized density is
\[
 \prod_{i=1}^n\lambda_i e^{-\lambda_i t_i}\one_{\{t_i\ge0\}}, 
\]
which implies \eqref{eq:extremizers}, with $T$ having columns
$v_i/\lambda_i$.
These laws attain equality for all $a,b>0$, as shown in the proof
of Theorem~\ref{thm:main}. 
\end{IEEEproof}

\begin{proposition}[Min-entropy equality]\label{prop:min-entropy-equality}
Fix $a,b>0$. A log-concave probability density $f$ attains equality
in \eqref{eq:main} at $p=\infty$ if and only if, for some $c\in\R^n$,
\begin{equation}\label{eq:homogeneous-extremizers}
 f(x)=Z^{-1}e^{-\psi(x-c)},\qquad 
Z=\int_{\R^n}e^{-\psi(x)}\dd x\in(0,\infty), 
\end{equation}
where $\psi:\R^n\to[0,\infty]$ is lower semicontinuous, convex,
$\psi(0)=0$, and $\psi(tx)=t\psi(x)$ for every $t>0$.
Equivalently, $\psi$ is the extended-valued Minkowski functional
of a convex body $K$ with $0\in K$; the origin may lie on
the boundary of $K$. An extremizer for one positive pair of weights
is an extremizer for every such pair. 
\end{proposition}

\begin{IEEEproof}
	By the scaling property of entropy, we may assume that $a+b=1$.
	Choose a mode $c$ of $f$, which exists by upper semicontinuity
	and Lemma~\ref{lem:integrability}, and set $M=f(c)=\|f\|_\infty$ and $q(x)=f(c+x)/M$. 
	Then $q(0)=1$, $0\le q\le1$, and $M\int q=1$.
	Let $g$ be the continuous density of $aX-bX'$. Since $q(0)=1$, log-concavity yields 
	$q(tz)\ge q(z)^t$ for $0<t<1$. By a change of variables, it holds that 
	\begin{equation}\label{eq:mode-comparison}
		\begin{aligned}
			g((a-b)c)=M^2\int q(bz)q(az)\dd z\ge M^2\int q(z)\dd z=M,
		\end{aligned}
	\end{equation}
	which proves the min-entropy bound.
	
	If equality holds, then $\|g\|_\infty=M$, so equality holds
	in \eqref{eq:mode-comparison}. Set $s=\max(a,b)<1$.
	Log-concavity at the origin yields 
	\[
	q(az)q(bz)\ge q(sz)^{1/s}\ge q(z).
	\]
	Equality of the integrals therefore yields
	$q(sz)=q(z)^s$ almost everywhere. Since both sides are upper
	semicontinuous log-concave functions with full-dimensional
	positivity sets, the identity holds everywhere.
	
	Set $\psi=-\log q$, with $-\log0=+\infty$.
	Then $\psi$ is lower semicontinuous, convex, and nonnegative,
	with $\psi(0)=0$ and $\psi(sz)=s\psi(z)$. Iterating this identity in both directions  yields 
	$\psi(s^jz)=s^j\psi(z)$ for every $j\in\mathbb Z$. For each $z$, convexity and $\psi(0)=0$ imply that
	$t\mapsto\psi(tz)/t$ is nondecreasing on $(0,\infty)$.
	Its values agree at all $s^j$, so it is constant.
	Thus $\psi(tz)=t\psi(z)$ for every $t>0$, including when
	$\psi(z)=+\infty$. Since $M\int q=1$, this gives
	\eqref{eq:homogeneous-extremizers}, with $Z=M^{-1}$.
	
	Conversely, let $f$ have the stated form and translate $c$
	to zero. By convexity and homogeneity,
	\[
	\psi(bx+d)+\psi(ax-d)
	\ge \psi((a+b)x)=\psi(x),
	\qquad x,d\in\R^n.
	\] 
	By a change of variables, it holds that 
	\[
	\begin{aligned}
		g(d)=Z^{-2}\int
		e^{-\psi(bx+d)-\psi(ax-d)}\dd x\leq Z^{-2}\int e^{-\psi(x)}\dd x=Z^{-1}.
	\end{aligned}
	\]
	At $d=0$, equality holds by homogeneity. Hence
	$\|g\|_\infty=Z^{-1}=\|f\|_\infty$.
	This proves equality for every normalized pair, and scaling
	restores arbitrary positive weights.

	For the geometric formulation, set $K=\{\psi\le1\}$.
	This set is closed, convex and contains zero. By homogeneity and full-dimensionality of $\dom\psi$,
	$K$ has nonempty interior; its boundedness follows from
	the exponential bound in Lemma~\ref{lem:integrability}. Then $K$ is a convex body, and $\{\psi\le t\}=tK$
	for every $t>0$ by homogeneity. Hence $\psi$ is the extended-valued Minkowski functional of $K$:
	\[
	\psi(x)=\inf\{t>0:x\in tK\},
	\qquad \inf\varnothing=+\infty.
	\] 
	Conversely, the extended-valued Minkowski functional of any
	convex body containing zero has the stated properties, with
	$Z=n!|K|\in(0,\infty)$ by integration over its sublevel sets. 
\end{IEEEproof} 

\begin{IEEEproof}[Proof of Theorem~\ref{thm:Shannon} and Corollary~\ref{cor:Ruzsa}]
The entropy power inequality \cite{Stam,Blachman,CoverThomas} gives
$N(X-X')\ge N(X)+N(-X')=2N(X)$. Equality requires Gaussian
summands with proportional covariance matrices. Here the two
covariance matrices agree; they are finite and positive definite by
full-dimensionality and Lemma~\ref{lem:integrability}. The upper
bound and its equality cases follow from Theorems~\ref{thm:main}
and \ref{thm:equality}. Equation~\eqref{eq:Ruzsa-definition}
then gives Corollary~\ref{cor:Ruzsa}.
\end{IEEEproof}

\section{Shannon consequences}\label{sec:Shannon-consequences}
\subsection{Signed sums of independent copies}
Yu \cite[Theorem 1 and Lemma 1]{Yu} proved the
corresponding entropy comparison under majorization for real-valued
random variables, using an argument originating in
Karlin and Rinott \cite{KarlinRinott}. For the comparison needed
here, a direct coupling gives the same argument for random vectors.

\begin{corollary}[Signed sums of independent copies]\label{cor:signed-sums}
Let $m\ge1$, and let $X_1,\ldots,X_m$ be independent copies of a
random vector $X$ with a log-concave probability density on $\R^n$.
For real coefficients $c_1,\ldots,c_m$ that are not all zero, set
\[
 \begin{gathered}
 S_c=\sum_{i=1}^m c_iX_i,\qquad
 P=\sum_{c_i>0}c_i,\qquad N=\sum_{c_i<0}(-c_i),\qquad 
 L=P+N=\sum_{i=1}^m|c_i|>0.
 \end{gathered}
\]
For an independent copy $X'$ of $X$,
\begin{equation}\label{eq:signed-sums}
 h(S_c)\le h(PX-NX')\le h(X)+n\log L.
\end{equation}
Equality in the first inequality holds if and only if at most one
coefficient is positive and at most one is negative. Consequently,
$h(S_c)=h(X)+n\log L$ holds precisely when either only one
coefficient is nonzero, or exactly two coefficients are nonzero,
have opposite signs, and the law of $X$ is of the form
\eqref{eq:extremizers}.
\end{corollary}

\begin{IEEEproof}
	If $P>0$, choose an index $I$ among the positive coefficients
	with probabilities $c_i/P$. If $N>0$, choose an index $J$
	among the negative coefficients with probabilities $(-c_j)/N$.
	Take the indices independent of each other and of the copies,
	and set $T=PX_I-NX_J$, omitting the corresponding term when
	$P$ or $N$ is zero. Independence of the copies and disjointness of the index sets yield 
	\[
	T\stackrel{\mathrm d}=PX-NX',\qquad
	\E[T\mid X_1,\ldots,X_m]=S_c.
	\]
	
	Let $u$ and $q$ be the densities of $S_c$ and $T$, respectively,
	and put $W=-\log q$, with value $+\infty$ where $q=0$.
	Both densities are log-concave, so their entropies are finite
	by Lemma~\ref{lem:integrability}.
	Since $W$ is convex and bounded below, by conditional Jensen's
	inequality, 
	$\E W(S_c)\le\E W(T)=h(T)<\infty$, 
	which yields $\int u|\log q|<\infty$. Hence
	\[
	h(S_c)+\KL(u\Vert q)=\E W(S_c)\le h(T).
	\]
	which implies the first inequality in \eqref{eq:signed-sums}. The second follows from
	Theorem~\ref{thm:main} when $P,N>0$, and from
	\eqref{eq:affine-entropy} otherwise.
	
	If $h(S_c)=h(T)$, then $\KL(u\Vert q)=0$, so $S_c$ and $T$
	have the same law. Set $v=\E|X-\E X|^2$, which is positive
	and finite by Lemma~\ref{lem:integrability} and full-dimensionality.
	Comparing centered second moments yields 
	\[
	v\sum_{i=1}^m c_i^2=v(P^2+N^2).
	\]
	Since $v>0$, this holds if and only if at most one coefficient
	of each sign is nonzero. Conversely, that condition makes
	$S_c$ and $T$ equal in distribution.
	The final equality assertion follows from
	Theorem~\ref{thm:equality} and \eqref{eq:affine-entropy}.
\end{IEEEproof} 

\subsection{Independent and displacement interpolation}\label{sec:interpolation}
The two-copy inequality also admits a comparison with optimal
transport. Let $\mu$ be the law of $X$, with a log-concave density
on $\R^n$, and let $\check\mu$ be the law of $-X$. 
Let $\pi_{\mathrm{opt}}$ minimize
$\int |x-y|^2\dd\pi(x,y)$ over couplings $\pi$ with marginals
$\mu$ and $\check\mu$. This is the
quadratic optimal transport problem. 
For $0<t<1$, define $ A_t(x,y)=(1-t)x+ty$ and 
\begin{equation*}
	 \mu_t^{\mathrm{ind}}=(A_t)_\#(\mu\otimes\check\mu),\qquad
	\mu_t^{\mathrm{opt}}=(A_t)_\#\pi_{\mathrm{opt}}.
\end{equation*}
These are the laws of $A_t(X,Y)$ under the independent and
optimal couplings, respectively. 
The second curve is the displacement interpolation between
$\mu$ and $\check\mu$. Then
\begin{equation}\label{eq:OT-sandwich}
 h(\mu_t^{\mathrm{ind}})\le h(\mu)\le h(\mu_t^{\mathrm{opt}}).
\end{equation}
Indeed, the first inequality is Theorem~\ref{thm:main} at $p=1$
with $a=1-t$ and $b=t$. For the second, McCann's displacement
convexity theorem \cite[Theorem 2.2]{McCann} applied to
$r\mapsto r\log r$ yields 
\[
 h(\mu_t^{\mathrm{opt}})
 \ge (1-t)h(\mu)+t h(\check\mu)=h(\mu).
\]

Thus independence and quadratic optimality give opposite entropy
comparisons for the same map $A_t$ and the same marginals.
Equality on the left of \eqref{eq:OT-sandwich} is characterized by
Theorem~\ref{thm:equality}. If $\mu$ is centrally symmetric
about $c$, the translation coupling $Y=X-2c$ is optimal,
so the optimal interpolants are translates of $\mu$ and
equality holds on the right. 

The independent coupling need not minimize the entropy of
the interpolant. Indeed, for jointly Gaussian $X,Y$ with
standard normal marginals and correlation $\rho\in(-1,0)$,
the midpoint variance is $(1+\rho)/2<1/2$, giving lower
entropy than the independent midpoint. 

\section{Orders below one}\label{sec:critical}
Finally, we turn to orders $0<p<1$ and prove
Theorem~\ref{thm:critical}. By \eqref{eq:affine-entropy}, it suffices to consider the standard simplex. Let $\Sigma_n=\{x\in[0,\infty)^n:\sum_i x_i\le1\}$
be the standard $n$-dimensional simplex, and let $U,U'$
be independent and uniformly distributed on $\Sigma_n$. Let $f_n=|\Sigma_n|^{-1}\one_{\Sigma_n}$ be the density
of $U$, and let $g_n$ be the density of $U-U'$. We compute $g_n$ through intersection volumes; see also \cite{RS,AGGJV}. 

Write $r_+=\max\{r,0\}$ for the positive part of a real
number $r$. For $z=(z_1,\ldots,z_n)\in\R^n$, set
\[
\rho(z)=\max\left\{\sum_i(z_i)_+,\ \sum_i(-z_i)_+\right\}.
\]
For $\rho(z)<1$, translating $\Sigma_n\cap(z+\Sigma_n)$
by $-((z_i)_+)_i$ gives $(1-\rho(z))\Sigma_n$;
otherwise, the intersection has zero volume.
The convolution identity then yields 
\begin{equation}\label{eq:simplex-density}
	\begin{aligned}
		g_n(z)=(f_n*\widetilde f_n)(z)
		=\frac{|\Sigma_n\cap(z+\Sigma_n)|}{|\Sigma_n|^2}=\frac1{|\Sigma_n|}(1-\rho(z))_+^n.
	\end{aligned}
\end{equation} 
The set $\{\rho\le1\}$ is $\Sigma_n-\Sigma_n$.
Since $\Sigma_n$ is a simplex, equality in
\eqref{eq:classical-RS} gives
\begin{equation}\label{eq:simplex-volume}
	|\Sigma_n-\Sigma_n|=\binom{2n}{n}|\Sigma_n|.
\end{equation}

Let $\Gamma(s)=\int_0^\infty t^{s-1}e^{-t}\dd t$, $s>0$,
denote Euler's gamma function. For $r>0$, by homogeneity, it holds that 
$|\{\rho\le r\}|=r^n|\Sigma_n-\Sigma_n|$. 
Since $\int f_n^p=|\Sigma_n|^{1-p}$, radial integration
in \eqref{eq:simplex-density}, together with
\eqref{eq:simplex-volume}, yields, for every $p>0$,
\begin{equation}\label{eq:simplex-moment}
	\begin{aligned}
		\frac{\int g_n^p}{\int f_n^p}=\binom{2n}{n}\,n\int_0^1
		r^{n-1}(1-r)^{np}\dd r=\frac{\Gamma(2n+1)\Gamma(np+1)}
		{\Gamma(n+1)\Gamma(n(p+1)+1)}, 
	\end{aligned}
\end{equation}
where the last equality follows from the beta integral formula. 

\begin{IEEEproof}[Proof of Theorem~\ref{thm:critical}]
For fixed $0<p<1$, applying Stirling's formula in
\eqref{eq:simplex-moment}, we obtain 
\begin{equation}\label{eq:simplex-limit}
 \begin{aligned}
 \lim_{n\to\infty}\frac{h_p(U-U')-h_p(U)}n=\frac{2\log 2+p\log p-(p+1)\log(p+1)}{1-p}.
 \end{aligned}
\end{equation}
The limit in \eqref{eq:simplex-limit} exceeds $\log 2$, proving \eqref{eq:failure} for all sufficiently large $n$. 
\end{IEEEproof} 

\begin{remark}
	In particular, at $p=1/2$, the bound already fails in dimension three.
	Indeed, \eqref{eq:simplex-moment} with $n=3$ gives
	\[
	h_{1/2}(U-U')-h_{1/2}(U)
	=2\log\frac{64}{21}>3\log 2.
	\]
\end{remark} 

\section{Open problems}\label{sec:open-problems}
We distinguish questions suggested by the present results from
two further directions proposed to us by Jiange Li after the
difference inequality was obtained: positive sums and
integer-valued log-concave laws. Some related problems and partial
results already occur in the literature, as indicated below.

\subsection{Direct extensions}\label{sec:open-extensions}
\noindent\emph{Two different one-dimensional laws.}
For independent real-valued random variables $X,Y$ with log-concave
densities, is it true that
\begin{equation}\label{eq:open-two-laws}
 e^{h(X-Y)}\le e^{h(X)}+e^{h(Y)}?
\end{equation}
Since reflection preserves entropy and log-concavity, this is
equivalent to asking whether the same bound holds for $X+Y$.
Theorem~\ref{thm:main} proves \eqref{eq:open-two-laws} for positive
dilations and translations of a common law, with equality when
the common law is a one-sided exponential law.

This question is related to the entropy-norm conjecture of
Ball, Nayar and Tkocz \cite[Conjecture 1]{BNT}, which assumes
joint central symmetry and allows dependence. Here independence
is assumed, while no symmetry is required.
The higher-dimensional analogue, with $e^{h(\cdot)/n}$ in place
of $e^{h(\cdot)}$, already fails in $\R^2$. Indeed, independent
centered Gaussian vectors with covariance matrices
$\operatorname{diag}(1,\varepsilon^2)$ and
$\operatorname{diag}(\varepsilon^2,1)$ give a counterexample
for sufficiently small $\varepsilon>0$. 

\medskip
\noindent\emph{Orders below one in a fixed dimension.}
For $n\ge1$ and $0<p<1$, define
\begin{equation}\label{eq:open-Cnp}
 C_{n,p}=\sup_X\{h_p(X-X')-h_p(X)\},
\end{equation}
where the supremum is over log-concave probability densities on
$\R^n$ and $X'$ is an independent copy of $X$. Exponential products
and the entropy comparisons in \cite{BM11,FMW} give
\[
 n\log 2\le C_{n,p}\le n\frac{\log p}{p-1}.
\]
What is $C_{n,p}$, and which laws attain or approach it?
Theorem~\ref{thm:critical} makes the lower bound strict for each
fixed $p<1$ in sufficiently high dimension. For a fixed $n$, do the
orders satisfying $C_{n,p}=n\log 2$ form an interval adjacent to $1$?

\subsection{Positive sums and discrete laws}\label{sec:open-sums-discrete}
\noindent\emph{Positive sums.}
For independent copies $X,X'$ of a log-concave random vector in
$\R^n$, Fu and Li \cite[Conjecture 1.1]{FuLi} state the folklore
conjecture
\begin{equation}\label{eq:open-positive-sum}
 h(X+X')-h(X)\le n\gamma_{\!E},
\end{equation}
where $\gamma_{\!E}=\lim_{m\to\infty}(\sum_{k=1}^m k^{-1}-\log m)$
is the Euler--Mascheroni constant. For independent standard
exponential variables $E,E'$, direct integration gives
$h(E+E')-h(E)=\gamma_{\!E}$; hence the affine images of exponential
products in \eqref{eq:extremizers} are candidate extremizers.
Fu and Li \cite[Theorem 3.4]{FuLi} prove the sharp one-dimensional
min-entropy analogue. Since $\gamma_{\!E}<\log 2$, the conjectured
Shannon bound improves on Corollary~\ref{cor:signed-sums}.

\medskip
\noindent\emph{Integer-valued log-concave laws.}
An integer-valued random variable $X$ with probability mass
function $r_k=\mathbb P(X=k)$ is called log-concave if
$\{k\in\mathbb Z:r_k>0\}$ is an interval of consecutive integers
and $r_k^2\ge r_{k-1}r_{k+1}$ for every $k\in\mathbb Z$;
see \cite[Definitions 2.3 and 2.4]{MPC} and
\cite[Section 4]{FuLi}.
Its discrete Shannon entropy is
$H(X)=-\sum_{k\in\mathbb Z}r_k\log r_k$. For independent copies $X,X'$ with such a law, what are the
sharp upper bounds on $H(X-X')-H(X)$ and $H(X+X')-H(X)$?
In particular, does
\begin{equation}\label{eq:open-discrete-difference}
	H(X-X')-H(X)\le\log 2
\end{equation}
hold for every such law?
For independent geometric variables $Z_q,Z_q'$ with probability
mass $(1-q)q^k$ at $k=0,1,\ldots$, where $0<q<1$,
direct summation yields 
\[
H(Z_q-Z_q')-H(Z_q)
=\log(1+q)-\frac{q}{1+q}\log q
\uparrow\log 2
\quad(q\uparrow1), 
\]
which implies the constant $\log 2$ is optimal if \eqref{eq:open-discrete-difference} holds. Melbourne and Palafox-Castillo \cite[Theorem 2.11]{MPC} prove the corresponding difference bound for R\'enyi orders
$p\ge2$ when the probability masses are monotone on their support. Fu and Li \cite[Theorem 4.3]{FuLi} prove the sharp min-entropy bound for positive sums under the same assumption. The questions above concern Shannon entropy for general
log-concave masses. 

\subsection{Stability}\label{sec:open-stability}
For fixed dimension and positive normalized weights, can the two
nonnegative terms in the Shannon deficit of
Theorem~\ref{thm:information} control, after an invertible affine
normalization, the deviation of the potential from an affine
function and the deviation of its support from a simplicial cone?
Such bounds would quantify the equality classification for
exponential products.

\section*{Acknowledgment}
The authors thank Jiange Li for valuable discussions and suggestions,
in particular for proposing the positive-sum and integer-valued
log-concave directions after the difference inequality had been
obtained, and for informing them of his joint work with Zhen Fu.
Zhuo-Nan Zhu thanks Xiang-Dong Li for the talk at the University of
Science and Technology of China that introduced him to the
entropy-difference problem.

\medskip
\noindent\emph{Use of generative AI.}
OpenAI's ChatGPT and Codex, and Moonshot AI's Kimi, were used extensively during
research and writing for mathematical exploration, testing
candidate proof strategies and counterexamples, assistance with
derivations and proof details, literature searches, and drafting
and revision throughout
Sections~\ref{sec:introduction}--\ref{sec:open-problems}.
The motivation, central ideas and proof framework for the main
theorems were developed by the authors. AI-generated suggestions
were evaluated and revised during preparation of the manuscript.
The authors have reviewed and verified all mathematical arguments
in the final manuscript and take full responsibility for its
mathematical content, references, interpretations and presentation.

\end{document}